\documentclass[10pt]{amsart}

\usepackage{amssymb,latexsym,amscd}

\numberwithin{equation}{section}

\long\def\eatit#1{}

\newtheorem{thm}{Theorem}[section]
\newtheorem{prop}[thm]{Proposition}
\newtheorem{lem}[thm]{Lemma}
\newtheorem{cor}[thm]{Corollary}

\theoremstyle{definition}

\newtheorem{Ex}[thm]{Example}
\newtheorem{Rmk}[thm]{Remark}
\newtheorem{Ques}[thm]{Question}
\newtheorem{Conj}[thm]{Conjecture}

\newcommand{\pr}[1]{{{\bf P}^{#1}}}

\newcommand{\prfend}{\hbox to7pt{\hfil}
\par\vskip-\baselineskip\hbox to\hsize
{\hfil\vbox {\hrule width6pt height6pt}}\vskip\baselineskip}

\newcommand{\OO}{{\mathcal{O}}}

\begin{document}


\title{The resurgence of ideals of points and the containment problem}

\author{Cristiano Bocci \& Brian Harbourne}

\address{Cristiano Bocci\\ 
Dipartimento di Scienze Matematiche e Informatiche "R. Magari"\\
Universit\`a degli Studi di Siena\\
Pian dei mantellini, 44\\
53100 Siena, Italy}
\email{bocci24@unisi.it}

\address{Brian Harbourne\\
Department of Mathematics\\
University of Nebraska\\
Lincoln, NE 68588-0130 USA}
\email{bharbour@math.unl.edu}

\date{June 21, 2009}

\thanks{Acknowledgements: This reseach was partially supported by GNSAGA of INdAM (Italy)
and by the NSA}

\begin{abstract} We relate properties of linear systems on $X$ to the
question of when $I^r$ contains $I^{(m)}$ in the case that
$I$ is the homogeneous ideal of a finite set of distinct points $p_1,\ldots,p_n\in\pr 2$,
where $X$ is the surface obtained by blowing up the points. 
We obtain complete answers for when $I^r$ contains $I^{(m)}$ when
the points $p_i$ lie on a smooth conic, or when the points are general and $n\le 9$.
\end{abstract}


\maketitle


\ \vskip-.5in
\ 

\section{Introduction}\label{intro}

Let $I$ be a homogeneous ideal in a polynomial ring $k[x_0,\ldots,x_N]=k[\pr N]=R$
over an algebraically closed field $k$ of arbitrary characteristic.
Given a homogeneous ideal $I\subseteq R$, the $m$th {\it symbolic power}
of $I$ is the ideal 
$$I^{(m)}=R\cap(\cap_{P\in{\rm Ass}(I)} (I^mR_P)).$$
For an ideal of the kind we will mostly be interested in here,
i.e., an ideal of the form $I=\cap_i (I(p_i)^{m_i})$ where
$p_1,\ldots,p_n$ are distinct points of $\pr N$, $I(p_i)$ is the ideal 
generated by all forms vanishing at $p_i$ and each $m_i$
is a non-negative integer, $I^{(m)}$ turns out to be $\cap_i (I(p_i)^{mm_i})$. 
If $I^m$ is the usual power, then there is clearly a containment $I^m\subseteq I^{(m)}$
and indeed, for $0\neq I\subsetneq R$, $I^r\subseteq I^{(m)}$ holds if and only if $r\ge m$
\cite[Lemma 8.1.4]{refPSC}.
A much more difficult problem is to determine when there are containments
of the form $I^{(m)}\subseteq I^r$.
The results of \cite{refELS} and \cite{refHH1} show that
$I^{(m)}\subseteq I^r$ holds whenever $m\ge Nr$. 
The second author has proposed the following conjecture \cite[Conjecture 8.4.2]{refPSC}:

\begin{Conj} Let $I\subseteq k[\pr N]$ be a homogeneous ideal.
Then $I^{(m)}\subseteq I^r$ if $m\ge rN-(N-1)$.
\end{Conj}

This conjecture has been verified in a range of examples (such as when $I$
is the radical ideal of a finite set of generic points in $\pr2$ \cite{refBH}, or 
when $I$ is a radical ideal defining a finite set of points
in $\pr N$ and $r$ is a power of the characteristic 
when ${\rm char}(k)>0$ \cite[Example 8.4.4]{refPSC},
or when $I$ is a monomial ideal \cite[Example 8.4.5]{refPSC}).
Even if this conjecture is true, there is still the question
of determining for any given ideal $I$ and each $r$ what the least $m$ is for which
$I^{(m)}\subseteq I^r$ holds. An asymptotic version of this problem
is to determine the least real number $\rho(I)$, called the {\it resurgence} 
of $I$ \cite{refBH}, such that $m>r\rho(I)$ implies $I^{(m)} \subseteq I^r$. 
Thus the result of \cite{refELS} and \cite{refHH1} shows that $\rho(I)\le N$.
This is optimal in the sense that
for any real number $c<N$ \cite{refBH} constructs an ideal $I$ with
$\rho(I)>c$. 

In this paper we will in a range of cases address both the problem of computing the resurgence
and the problem of finding all $m$ and $r$ such that containment holds,
for ideals defining fat point subschemes of $\pr2$, whose
general definition we now recall.
Given distinct points $p_i\in\pr N$ and non-negative integers $m_i$,
we denote by $Z=m_1p_1+\cdots+m_np_n\subset \pr N$
the subscheme (known as a {\it fat point\/} subscheme)
defined by $I(Z)=\cap_i I(p_i)^{m_i}$, where $I(p_i)$
is the ideal generated by all forms which vanish at $p_i$. 

We also recall various additional definitions we will need.
Given any homogeneous ideal $(0)\neq I\subsetneq k[\pr N]$, 
following \cite{refBH} we use the following notation:
\begin{itemize}
\item[$\bullet$] $\alpha(I)$ is the degree of a homogeneous generator of $I$ of least degree
(equivalently, it is the $M$-adic order of $I$, i.e., the largest $t$ such that $M^t$ contains $I$, 
where $M$ is the maximal homogeneous
ideal of $k[\pr N]$),
\item[$\bullet$] $\gamma(I)=\lim_{m\to\infty}\alpha(I^{(m)})/m$, 
\item[$\bullet$] $\rho(I)$ is the supremum of all
ratios $m/r$ such that $I^{(m)}\not\subseteq I^r$,
\item[$\bullet$] ${\rm reg}(I)$ is the Castelnuovo-Mumford regularity; if $I$
is the ideal of a fat point subscheme, it is the least degree $t>0$ such that
$\dim(R/I)_t=\dim(R/I)_{t-1}$, where $\dim (R/I)_i=\dim R_i - \dim I_i$ and 
$I_i$ (resp.\ $R_i$) is the vector space span in $I$ (resp.\ $R$) 
of the forms of degree $i$ in $I$ (resp.\ $R$).
\end{itemize}

The quantities above are related. We note that by \cite{refBH} we have
$\rho(I)\ge \alpha(I)/\gamma(I)\ge  \alpha(I^m)/\alpha(I^{(m)})$ for all $m\ge 1$.
(Note that $\gamma(I)\ge1$ since $I\neq k[\pr N]$, by \cite[Lemma 8.2.2]{refPSC}.
For example, if $Z\neq 0$ is a fat point subscheme, then 
$Z$ contains some point $p$ as a subscheme.
Thus $\alpha(I(mp))\le \alpha(I(mZ))$, so using the easy fact that $m=\alpha(I(mp))$  
we see $\gamma(I(Z))\ge 1$. It is also true that $\rho(I(Z))\ge 1$ when $Z\ne 0$. To see this,
note that $\alpha(I(mZ))\le m\alpha(I(Z))$ for all $m\ge 1$, so $1\le \alpha(I(Z))/\gamma(I(Z))$;
now use $\alpha(I(Z))/\gamma(I(Z))\le \rho(I(Z))$ \cite{refBH}.)
Of course, $\alpha(I^m)/\alpha(I^{(m)})$ gives a measure of 
the growth that occurs when an ordinary power $I^m$ is replaced by 
a symbolic power $I^{(m)}$, but it is possible that 
$\alpha(I^m)/\alpha(I^{(m)})=1$ even though $\rho(I)>1$ and
$I^m\subsetneq I^{(m)}$ for all $m>1$ (see Lemma \ref{freepartptsonsmoothconic}
and Theorem \ref{rhoptsonsmoothconic}).
Thus $\rho(I)$ gives an asymptotic measure of additional growth
not detected by $\alpha(I^m)/\alpha(I^{(m)})$.

An additional important relationship was found in \cite{refBH}. If $I\subset k[\pr N]$
is the ideal of a 0-dimensional subscheme, then 
$\alpha(I)/\gamma(I)\le \rho(I)\le {\rm reg}(I)/\gamma(I)$,
and hence when $\alpha(I)={\rm reg}(I)$, we have an exact determination
$\rho(I)=\alpha(I)/\gamma(I)$. In this situation,
as an immediate consequence of  \cite[Lemmas 2.3.2(a) and 2.3.4]{refBH}, 
we also have the following solution for the containment problem:

\begin{cor}\label{a=rCor}
Assume $I\subset k[\pr N]$ is the ideal of a 0-dimensional subscheme of $\pr N$,
and that $\alpha(I)={\rm reg}(I)$. Then $I^{(m)}\subseteq I^r$ if and only if
$\alpha(I^{(m)})\ge r\alpha(I)$.
\end{cor}

Unfortunately, $\alpha(I)={\rm reg}(I)$ often does not hold, and even when it does
it can be very hard to compute $\alpha(I^{(m)})$ when $m$ is large.
For example, say $I$ is the ideal defining $n$ generic points in
$\pr2$. Then the value of $\gamma(I)$ is not in general known for $n>9$
 (it is conjectured to be $\gamma(I)=\sqrt{n}$ \cite[Section 1.3]{refBH} when $n>9$)
nor is $\alpha(I^{(m)})$ in general known for $n>9$ (by the SHGH Conjecture
of Segre-Harbourne-Gimigliano-Hirschowitz \cite{refS, refH4, refG, refHi}, 
it is conjectured that $\alpha(I^{(m)})$ is the least $t$ such that $\binom{t+2}{2}>n\binom{m+1}{2}$
when $n>9$). For $n\le9$ generic points in $\pr2$, 
$\alpha(I^{(m)})$ has been known for a long time \cite{refC} 
and one can also compute $\gamma(I)$ in these cases. When $n$ is a square
$\gamma(I)$ is known \cite[Section 1.3]{refBH} and recent work has also determined
$\alpha(I^{(m)})$ when $n$ is a square \cite{refCM, refE, refR}, verifying the SHGH Conjecture.

On the other hand, it is easy to specify when $\alpha(I)={\rm reg}(I)$ for the ideal $I$ of $n$ generic
points of $\pr2$: this occurs exactly when $n=\binom{s+1}{2}$ for some $s\ge1$,
in which case $\alpha(I)={\rm reg}(I)=s$.
Thus Corollary \ref{a=rCor} gives a complete explicit solution (i.e.,
$\rho(I)=\alpha(I)/\gamma(I)$ and $I^{(m)}\subseteq I^r$ if and only if
$\alpha(I^{(m)})\ge r\alpha(I)$) to the containment problem for the ideal $I$ of $n$ generic points
in those cases for which $\alpha(I)={\rm reg}(I)$ such that $\gamma(I)$ is known and $\alpha(I^{(m)})$
is known for all $m\ge1$; i.e., 
when $n$ is either $1, 3, 6$ or $n$ is any square which is at the same time a binomial coefficient.
(This happens infinitely often, starting with $s=1, 8, 49, 288,\ldots$ \cite[Section 1.3]{refBH}.)

Here we obtain results for ideals $I$ defining
points $p_1,\dots,p_n\in \pr 2$ in a range of cases for which $\alpha(I)\neq{\rm reg}(I)$,
by applying geometrically relevant properties of linear systems
on the variety $X$ obtained by blowing up the points $p_i$.
Indeed, we completely determine the set of ordered pairs $(m,r)$ for which 
$I^{(m)} \subseteq I^r$ holds in case $I\subset k[\pr 2]$ 
is the ideal of a finite set $p_1,\dots,p_n\in \pr 2$ of points
when either the points are general and $n\le 9$, or the points lie
on an irreducible conic in $\pr 2$ for any $n$. At the same time we determine $\rho(I)$ in these
cases. 

\section{Background}\label{bg}

We now recall or prove results we will need later.

In this section, let $\pi: X\to \pr 2$ be obtained by blowing up distinct points
$p_1,\ldots,p_n\in \pr 2$. Let $L$ be the pullback via $\pi$ to $X$
of a general line in $\pr 2$
and let $E_i$ be the the blow up of $p_i$.
Given the fat point subscheme $Z=m_1p_1+\cdots+m_np_n$,
the ideal $I(Z)$ is homogeneous. Let $I(Z)_t$ be the homogeneous component of $I(Z)$
of degree $t$; i.e., $I(Z)_t$ is the $k$-vector space
span of the forms in $I(Z)$ of degree $t$. Then we have a natural identification of
$I(Z)_t$ with $H^0(X, \OO_X(F_t(Z)))$, where $F_t(Z)$ denotes
$tL-m_1E_1-\cdots-m_nE_n$ (see, for example, \cite[Proposition IV.1.1]{refH7}). 

The linear equivalence 
classes of the divisors $L, E_1,\ldots,E_n$ give an orthogonal basis for the divisor class group
$\hbox{Cl}(X)$ of $X$ such that $-L^2=E_i^2=-1$. With respect to this basis, the
canonical class $K_X$ is $K_X=[-3L+E_1+\cdots+E_n]$.
For simplicity we will suppress the brackets when writing
the class $[aL-m_1E_1-\cdots-m_nE_n]$ of a divisor
$aL-m_1E_1-\cdots-m_nE_n$; in context the meaning will 
always be clear. Also, given a divisor $F$,
when indicating cohomology, we will 
for example often write $H^0(X, F)$ in place of $H^0(X, \mathcal{O}_X(F))$.

We will say that a divisor $F$ is {\it normally generated\/}
if the natural map 
$$H^0(X, F)^{\otimes s}\to H^0(X, sF)$$ 
is surjective for all $s\ge 1$. Note that we do not require that $F$ be ample.
Thus for example, $L-E_1$ is normally generated and nef but not ample.
(We recall that a divisor or divisor class $D$ is {\it nef} if $D\cdot C\ge0$
for every effective divisor $C$, and $D$ is ample if $D^2>0$ and $D\cdot C>0$
for every effective divisor $C$ \cite[Theorem V.1.10]{refHr}.)

Our hypothesis in this section that $Q=2L-E_1-\cdots-E_n$ be the class of an effective divisor
simply means that the points $p_i$ lie on a conic. Thus the results of this section
are useful for analyzing the case of points on conics, where we can hope to get
complete answers. However, when the conic is not irreducible (equivalently, not smooth),
this analysis requires a treatment of subcases (depending on how many points
are on each of the two lines of which the conic is composed). Our main interest in 
this paper is the case of $n\le 9$ general points. Since any $n\le 5$ general
points lie on a smooth conic, we will in the next section
apply the results of this section for points on a smooth conic
(and obtain for free a complete answer for any number of points 
on a smooth conic). However, the results of this section do not require
that the conic be smooth, so with a view to using the results below to analyze the case of
points on reducible conics in the future, we state our results here
assuming only that $Q=2L-E_1-\cdots-E_n$ is the class of an effective divisor.

\begin{lem}\label{ng} Assume $Q=2L-E_1-\cdots-E_n$ is the 
class of an effective divisor
on $X$  (i.e., the points $p_i$ lie on a plane conic, not necessarily smooth). 
If $F$ is nef, then the class of $F$ is the class of
an effective divisor, $h^1(X, F)=0$, $|F|$ is base point free
and $F$ is normally generated.
\end{lem}

\begin{proof} If $F$ is nef, then $F\cdot L\ge0$ since $L$ is effective 
(so $(K_X-F)\cdot L\le K_X\cdot L =-3$,
and thus $h^2(X,F)=h^0(X, K_X-F)=0$ since $L$ is also nef) and
$F^2\ge0$ \cite[Proposition II.3]{refH5}, but $-K_X=Q+L$ is effective,
so $-F\cdot K_X\ge0$, hence by Riemann-Roch, $h^0(X, F)\ge (F^2-F\cdot K_X)/2+1>0$, so
$F$ is effective. Moreover, $h^1(X, F)=0$ and $|F|$ is base point free 
by Lemma 3.1.1(b) of
\cite{refH2}. If in addition $F^2>0$ and $-K_X\cdot F\ge3$,
then $F$ is normally generated by Proposition 3.1 of \cite{refH1}.
But $-K_X=Q+L$, so $F$ nef means $-K_X\cdot F\ge L\cdot F\ge0$.
If $L\cdot F=0$, then it is easy to see that $F=0$, which 
trivially is normally generated.
If $L\cdot F=1$, then either $F=L$ or $F=L-E_i$ for some $i$,
and it is easy to see that in both cases $F$ is normally generated.
Finally, if $L\cdot F=2$, then up to reordering the points $p_i$, $F$ must be
either $2L$, $2L-E_1$, $\ldots$, $2L-E_1-\cdots-E_4$ or $2L-2E_1$
(since if $m_i>2$ for some $i$, then $F\cdot(L-E_i)<0$, but $L-E_i$ is nef,
and if $F=2L-E_1-\cdots-E_i$ for $i>4$, then $F^2<0$, while if
$F=2L-m_1E_1-\cdots-m_iE_i$ where $m_1>1$ and $m_2>0$, then
$F\cdot (L-E_1-E_2)<0$, which contradicts the assumption that $F$ is nef
since $L-E_1-E_2$ is linearly equivalent to an effective divisor).
In each case $F$ nef implies $-K_X\cdot F\ge3$ and $F^2>0$
unless either $F=2(L-E_i)$ or $F=2L-E_{i_1}-\cdots-E_{i_4}$, but in 
both cases $F$ is normally generated, since
a single point and also four points on an irreducible conic are complete intersections,
and hence in these cases each element of $|mF|$ is a sum of $m$ components,
each component being an element of $|F|$ (i.e., the map $H^0(X, F)^{\otimes m}\to H^0(X, mF)$
is onto).
\end{proof} 

\begin{lem}\label{FplusLlem} Assume $Q=2L-E_1-\cdots-E_n$ is the class of an effective divisor
on $X$. If $F$ is nef, then $H^0(X,F)\otimes H^0(X,L)\to H^0(X, F+L)$ is surjective.
\end{lem}

\begin{proof} This is Theorem 3.1.2 of
\cite{refH2}. See also Lemma 2.11 of \cite{refGuardoHar}.
\end{proof} 

If $F$ and $G$ are nef, it is not always true that
$H^0(X,F)\otimes H^0(X,G)\to H^0(X, F+G)$ is surjective,
even if $Q=2L-E_1-\cdots-E_n$ is the class of an effective divisor.
The next two results give special cases where
surjectivity does hold more generally than
what we have by Lemma \ref{FplusLlem}.

\begin{cor}\label{FplusLcor} Assume $Q=2L-E_1-\cdots-E_n$ is the class of an effective divisor
on $X$. Let $i,j,r,s$ be non-negative integers.
If $F$ is nef, then $H^0(X,iF+jL)\otimes H^0(X,rF+sL)\to H^0(X, (i+r)F+(j+s)L)$ is surjective.
\end{cor}

\begin{proof}
By Lemma \ref{ng}, 
$H^0(X,F)^{\otimes i+r}\to H^0(X,(i+r)F)$ is surjective and 
thus induces a surjective map
$H^0(X,F)^{\otimes i+r}\otimes H^0(X,L)^{\otimes j+s}\to H^0(X,(i+r)F)\otimes H^0(X,L)^{\otimes j+s}$,
and by Lemma \ref{FplusLlem}
$H^0(X,(i+r)F)\otimes H^0(X,L)^{\otimes j+s}\to H^0(X, (i+r)F+(j+s)L)$ is surjective,
hence $H^0(X,F)^{\otimes i+r}\otimes H^0(X,L)^{\otimes j+s}\to H^0(X, (i+r)F+(j+s)L)$ is surjective.
But this map factors as
$$
\noindent \begin{tabular}{clcl}
     & $H^0(X,F)^{\otimes i+r}\otimes H^0(X,L)^{\otimes j+s}$ & $\to$ & 
     $H^0(X,iF)\otimes H^0(X,jL)\otimes H^0(X,rF)\otimes H^0(X,sL)$ \\
$\to$ & $H^0(X,iF+jL)\otimes H^0(X,rF+sL)$                                 & $\to$ & $H^0(X, (i+r)F+(j+s)L)$ 
\end{tabular}
$$
hence $H^0(X,iF+jL)\otimes H^0(X,rF+sL) \to H^0(X, (i+r)F+(j+s)L)$ is surjective.
\end{proof} 

Let us say that a divisor class $F=iL-m_1E_1-\cdots-m_nE_n$ 
is {\it uniform} (some authors use the term {\it homogeneous} here) if $m_1=\cdots=m_n=m$ for some $m$.
Likewise we say that $Z=m_1p_1+\cdots+m_np_n$ is {\it uniform}
if $m_1=\cdots=m_n=m$ for some $m$.

\begin{prop}\label{UnifPlusNefLem} Assume $Q=2L-E_1-\cdots-E_n$ is the 
class of an effective divisor on $X$. 
If $F$ is nef and if $G$ is nef and uniform, then 
$H^0(X,F)\otimes H^0(X,G)\to H^0(X, F+G)$ is surjective.
\end{prop}

\begin{proof} If $n=0$, then $X=\pr 2$, so a nef divisor
must be of the form $iL$ for $i\ge 0$, while if $n=1$ a nef divisor
is of the form $i(L-E_1)+jL$ for non-negative $i$ and $j$.
Thus for $n\le 1$ the result follows by Corollary \ref{FplusLcor}. 
Also, if $G=iL$, then again the result follows by Corollary \ref{FplusLcor}.
Thus we may assume that $G=iL-m(E_1+\cdots+E_n)$ with $m>0$ and $n\ge 2$.
Note that $G-Q$ is nef (and uniform) in this case 
(and hence $h^1(X, \OO_X(G-Q))=0$ by Lemma \ref{ng}):
to see that $G-Q$ is nef, note that since $G$ is nef, we have 
$G\cdot (L-E_1-E_2)\ge0$ so $i\ge 2m$. Thus $G=mQ+(i-2m)L$
so $G-Q$ is effective and nefness follows if $(G-Q)\cdot D\ge 0$ for
every irreducible component $D$ of $C$ where $C$
is an effective divisor whose class is $Q$. If in fact $D\cdot(G-Q)<0$,
then $0>D\cdot(G-Q)=D\cdot((i-2m)L+(m-1)Q)$ implies $D\cdot Q<0$
hence $0>D\cdot(G-Q)\ge D\cdot G$.  Since $G$ is nef,
$D\cdot G< 0$ is impossible, hence $D\cdot(G-Q)<0$ is 
also impossible, so $G-Q$ is nef.

Now consider the exact sequence
$0\to \OO_X(F-Q)\to \OO_X(F)\to \OO_Q(F)\to 0$.
Since $-K_X=Q+L$ we have $h^1(X, \OO_X(F-Q))=h^1(X, \OO_X(F+L+K_X))$,
and by Serre duality this becomes $h^1(X, \OO_X(-(F+L)))$,
and, since $F+L$ is nef and big, $h^1(X, \OO_X(-(F+L)))=0$ by 
Ramanujam vanishing (for the characteristic $p$ version,
see Theorem 1.6 \cite{refT} or Theorem 2.8 of \cite{refH1}).
Thus $H^0(X, \OO_X(F))\to H^0(Q,\OO_Q(F))$
is surjective.

Consider the maps 
$$\mu_1:H^0(X,\OO_X(G-Q))\otimes H^0(X,\OO_X(F))\to H^0(X,\OO_X(F+G-Q)),$$
$$\mu_2:H^0(X,\OO_X(G))\otimes H^0(X,\OO_X(F))\to H^0(X,\OO_X(F+G)),$$ 
and
$$\mu_3:H^0(Q,\OO_Q(G))\otimes H^0(X,\OO_X(F))\to H^0(Q,\OO_Q(F+G)).$$
Taking global sections of $0\to \OO_X(G-Q)\to \OO_X(G)\to \OO_Q(G)\to 0$
we get a short exact sequence; tensor by $H^0(X,\OO_X(F))$ and map to 
the short exact sequence given by taking global sections of
$0\to \OO_X(F+G-Q)\to \OO_X(F+G)\to \OO_Q(F+G)\to 0$ to obtain the following diagram
(where $V_F$ denotes $H^0(X,\OO_X(F))$):

$${{\begin{array}{rcccl}
0  \to & H^0(X,\OO_X(G-Q))\otimes V_F  & \to H^0(X,\OO_X(G))\otimes V_F 
\to & H^0(Q,\OO_Q(G))\otimes V_F & \to  0 \cr
{}  {} & \downarrow \raise3pt\hbox to0in{$\scriptstyle\mu_1$\hss} & {}
 \downarrow\raise3pt\hbox to0in{$\scriptstyle\mu_2$\hss}  {}
& \downarrow\raise3pt\hbox to0in{$\scriptstyle\mu _3$\hss} \cr
0  \to & H^0(X,\OO_X(F+G-Q)) & \to  H^0(X,\OO_X(F+G)) \to &
H^0(Q,\OO_Q(F+G)))& \to  0 \cr
\end{array}}}$$

By the snake lemma we have an exact sequence
$\hbox{cok }\mu_1\to \hbox{cok }\mu_2\to \hbox{cok }\mu_3\to 0$.
Since $H^0(X, \OO_X(F))\to H^0(Q,\OO_Q(F))$ is surjective, $\mu_3$
and $\mu_4:H^0(Q,\OO_Q(G))\otimes H^0(Q,\OO_Q(F))\to H^0(Q,\OO_Q(F+G))$
have the same cokernel. We thus have an exact sequence
$\hbox{cok }\mu_1\to \hbox{cok }\mu_2\to \hbox{cok }\mu_4\to 0$.
By Lemma 2.11 of \cite{refGuardoHar}, $\hbox{cok }\mu_4=0$.
Thus $\hbox{cok }\mu_2=0$ if $\hbox{cok }\mu_1=0$, and this
last follows by induction on $m$, using the fact that $G-Q$ is 
nef and uniform, eventually reducing to the case $G=iL$ done above. 
\end{proof}

The notion of nefness in the context of the preceding results can be described algebraically,
denoting by $\omega(I)$ the largest degree among elements of any 
minimal set of homogeneous generators of $I$.

\begin{cor}\label{omegacor} Assume $2L-E_1-\cdots-E_n$
is the class of an effective divisor, and
let $0\neq Z=m_1p_1+\cdots+m_np_n$ be a fat point subscheme.
Then $F_t(Z)$ is nef if and only if $t\ge \omega(I(Z))$. 
\end{cor}

\begin{proof}
First suppose $F_t(Z)$ is not nef, hence there is an effective prime divisor $C$ with $F_t(Z)\cdot C<0$.
If $F_t(Z)$ is not the class of an effective divisor, then $I(Z)_t=(0)$, hence $t < \alpha(I)\le \omega(I)$.
If $F_t(Z)$ is the class of an effective divisor, then $C$ is a component of every element of $|F_t(Z)|$;
i.e., $C$ is a fixed component of $|F_t(Z)|$, and hence the zero-locus of 
$I(Z)_t$ is 1-dimensional so $I(Z)$ requires a generator in some degree bigger than $t$,
and again we have $t<\omega(I)$.
Conversely, by Lemma \ref{FplusLlem} we see that if $F_t(Z)$ is nef,
then $t\ge \omega(I(Z))$.
\end{proof}

\section{Points on Smooth Conics}\label{smoothconics}

We now focus on the case of points on a smooth plane conic.
We will use the following notation. Given $Z=p_1+\cdots+p_n$ 
contained in a smooth plane conic $Q'$ defined by a form $f$, 
and given any vector space $V$ of forms of equal degree, 
let $q(V)$ denote the largest exponent $e$ such that $f^e$ is a
factor of every element of $V$. 

\begin{lem}\label{freepartptsonsmoothconic} For $n\ge5$, assume the 
points $p_1,\ldots,p_n$
lie on a smooth conic curve $Q'$ defined by a degree 2 form $f$, 
hence $Q=2L-E_1-\cdots-E_n$ is the class of a smooth
curve with $Q^2<0$. Let $I=I(mZ)$ where $Z=p_1+\cdots+p_n$. 
Then $\alpha(I^r)=2mr$, and, for each $t\ge 2mr$, $I^r_{\ t}=f^qI((rm-q)Z)_{t-2q}$,
where $q=q(I^r_{\ t})$ is the minimum value of $q(I_{t_1})+\cdots+ q(I_{t_r})$
over all sums $t=t_1+\cdots+t_r$ with $t_i\ge2m$ for all $i$, and 
$q(I_{t_i})$ is the least $s\ge0$ such that $2(t_i-2s)\ge (m-s)n$.
\end{lem}

\begin{proof} 
Since $H=2L-E_1-\cdots-E_4$ is easily seen to be the class of a reduced irreducible divisor 
of self-intersection 0, it is nef, but $H\cdot ((2m-1)L-mE_1-\cdots-mE_n)<0$
so $\alpha(I)\ge2m$. Since clearly $f^m\in I$, we see  
$\alpha(I)=2m$, hence $\alpha(I^r)=r\alpha(I)=2mr$. Since $f^m\in I_{2m}$ and $I_{2m}$ can be identified
with $H^0(X, mQ)$, and since $mQ$ is a multiple of a prime divisor
of negative self-intersection, we see that $h^0(X, mQ)=1$ and thus $f^m$ spans $I_{2m}$.
But the greatest common factor of $I_{t}$ divides that of $I_{2m}$ for any $t\ge 2m$, so
for any $t\ge 2m$ the greatest common factor of $I_t$ is a power of $f$ and hence defines a divisor
$q(I_t)Q'$. Let $q'=q(I_t)$. Dividing out by this gcd leaves $I((m-q')Z)_{t-2q'}$,
which is fixed component free. Thus $F_{t-2q'}((m-q')Z)$ is nef and uniform.
In particular, $F_{t-2q'}((m-q')Z)\cdot Q\ge0$, so $2(t-2q')-(m-q')n\ge0$.
Let $F=F_{t-2(q'-1)}((m-(q'-1))Z)$.
Since $Q'$ is a common divisor for $I((m-(q'-1))Z)_{t-2(q'-1)}$,
$|F|$ is not fixed component free, hence
$F$ is not nef by Lemma \ref{ng}. Since 
$F=(t-2(m-(q'-1))L+(m-(q'-1))Q$,
the only prime divisor which $F$ could meet
negatively is $Q$, hence $0>F\cdot Q=2(t-2(q'-1))-(m-(q'-1))n$.
Thus whenever $t\ge 2m$,
$q(I_t)$ is the least $s\ge0$ such that $2(t-2s)\ge (m-s)n$.

If $t\ge\alpha(I^r)$, by definition of powers of an ideal (and the 
fact that $I_t=0$ for $t<\alpha(I)=2m$) we have
$I^r_{\ t}=\sum_{{\bf j}\in S} \Pi_{i=1}^r I_{j_i}$, where 
$S$ is the set of all sequences ${\bf j}=\{j_1,\ldots,j_r\}$, such that
$2m\le j_1\le \cdots\le j_r$ with $j_1+\cdots+j_r=t$
(we ignore sequences with $j_1<2m$ since $I_{j_1}=0$ for $j_1<2m$). 
Thus $q$ is the minimum
$q(\Pi_{i=1}^r I_{j_i})$ among all ${\bf j}_i\in S$. 
Of course, for ${\bf j}'\in S$, we have
$q(\Pi_{i=1}^r I_{j'_i})=\sum_i q(I_{j'_i})$. 
Let $q'_i=q(I_{j'_i})$ and, recycling notation, let $q'=\sum_iq'_i$. Then 
\[
\begin{split}
\Pi_{i=1}^r I_{j'_i}= \Pi_{i=1}^r f^{q'_i}I((m-q'_i)Z)_{j'_i-2q'_i} = & f^{q'}\Pi_{i=1}^r I((m-q'_i)Z)_{j'_i-2q'_i}\\
\subseteq & f^{q'}I(\sum_i(m-q'_i)Z)_{t-2q'}=f^{q'}I((rm-q')Z)_{t-2q'}.
\end{split}
\]
We can identify 
$\Pi_{i=1}^r I((m-q'_i)Z)_{j'_i-2q'_i}$ inside $I((rm-q')Z)_{t-2q'}$
with the image of 
\[
\begin{split}
\bigotimes_{i=1}^rH^0(X, \OO_X(F_{j'_i-2q'_i}((m-q'_i)Z)))\to & H^0(X, \OO_X(\sum_{i=1}^rF_{j'_i-2q'_i}((m-q'_i)Z)))\\
= & H^0(X, \OO_X(F_{t-2q'}((rm-q')Z))).
\end{split}
\]
By Proposition \ref{UnifPlusNefLem}, this map is surjective, 
and hence the inclusion $\Pi_{i=1}^r I_{j'_i}\subseteq f^{q'}I((rm-q')Z)_{t-2q'}$
above is an equality.
Since $q\le q'$, we thus see for every ${\bf j}'\in S$
that $\Pi_{i=1}^r I_{j'_i}=f^qf^{q'-q}I((rm-q')Z)_{t-2q'}$ is contained in 
$\Pi_{i=1}^r I_{j_i}=f^qI((rm-q)Z)_{t-2q}$ for that ${\bf j}=\{j_1,\ldots,j_r\}$ giving the minimum value $q$,
and hence 
$$f^qI((rm-q)Z)_{t-2q}=\Pi_{i=1}^r I_{j_i}\subseteq I^r_{\ t}=
\sum_{{\bf j}'\in S} \Pi_{i=1}^r I_{j'_i}\subseteq f^qI((rm-q)Z)_{t-2q}.$$
\end{proof}

\begin{cor}\label{containmentprop} Let $n\ge5$ and assume the points $p_1,\ldots,p_n$ 
lie on a smooth conic curve $Q'$
defined by a degree 2 form $f$. Let $I=I(sZ)$ for $Z=p_1+\cdots+p_n$
and $s>0$. Let $m\ge r$ and let $t\ge 2ms$.
Then $I^{(m)}_{\ t}\subseteq I^r_{\ t}$ if and only if $q(I^r_{\ t})\le q(I^{(m)}_{\ t})$.
\end{cor}

\begin{proof} Let $q_1=q(I^r_{\ t})$ and $q_2=q(I^{(m)}_{\ t})$.
Suppose we have $I^{(m)}_{\ t}\subseteq I^r_{\ t}$;
by Lemma \ref{freepartptsonsmoothconic}
we have
$$f^{q_2}I((ms-q_2)Z)_{t-2q_2}=I^{(m)}_{\ t}\subseteq I^r_{\ t}=
f^{q_1}I((rs-q_1)Z)_{t-2q_1},$$
which implies $f^{q_1}$ divides $f^{q_2}$ and hence that $q_1\le q_2$.

Conversely, $q_1\le q_2$ means we can divide out by
$f^{q_1}$, but $f^{q_2-q_1}I((ms-q_2)Z)_{t-2q_2}$
consists only of forms of degree $t-2q_1$
which vanish to order at least $ms-q_1$ at each point of $Z$, and hence
(since $ms-q_1\ge rs-q_1$)
is contained in the complete linear system $I((rs-q_1)Z)_{t-2q_1}$
of all forms of degree $t-2q_1$
vanishing to order at least $rs-q_1$ at each point of $Z$.
Multiplying back through by $f^{q_1}$ we have
$$I^{(m)}_{\ t}=f^{q_2}I((ms-q_2)Z)_{t-2q_2}\subseteq f^{q_1}I((rs-q_1)Z)_{t-2q_1}=I^r_{\ t}.$$
\end{proof}

\begin{lem}\label{formulaFORq} Assume the points $p_1,\ldots,p_n$ 
lie on a smooth conic curve $Q'$ with $n\ge5$. Let $I=I(Z)$ for $Z=p_1+\cdots+p_n$.
\begin{itemize}
\item[(a)] If $t\ge 2m$, then $q(I^{(m)}_{\ t}) = \hbox{max}\,(0,\,\lceil (mn-2t)/(n-4)\rceil)$.
\item[(b)] If $t\ge 2r$ and $n$ is odd, then $q(I^r_{\ t}) = \hbox{max}\,(0,\,\lceil (r(n+1)-2t)/(n-3)\rceil)$.
\end{itemize}
\end{lem}

\begin{proof} (a) Just apply Lemma \ref{freepartptsonsmoothconic} with $r=1$ and solve for $s$.

(b) Let $t=t_1+\cdots+t_r$ with $t_i\ge2$ for all $i$. By (a)
with $m=1$ we see that $q(I_{t_i})$ is 1 for $2\le t_i<n/2$ and 0 for $t>n/2$ ($t$ cannot equal $n/2$ 
since $n$ is odd). Thus
$q(I_{t_1}\cdots I_{t_r})=\sum_iq(I_{t_i})$ is the number $s$ of factors $I_{t_i}$ for which $t_i<n/2$.
Note that there is a product $I_{t_1}\cdots I_{t_r}$ having exactly $s$ factors $I_{t_i}$ with $2\le t_i<n/2$
if and only if $s$ satisfies $0\le s\le r$ and $2s+(r-s)\lceil n/2\rceil\le t$, so by Lemma \ref{freepartptsonsmoothconic}
$q(I^r_{\ t})$ is the least $s$ such that $0\le s\le r$ and $2s+(r-s)\lceil n/2\rceil\le t$. Solving $2s+(r-s)\lceil n/2\rceil\le t$
for $s$ using  $\lceil n/2\rceil =(n+1)/2$ gives $\lceil (r(n+1)-2t)/(n-3)\rceil\le s$. 
We claim that $u=\hbox{max}\,(0,\,\lceil (r(n+1)-2t)/(n-3)\rceil)$
is the least $s$ such that $0\le s\le r$ and $2s+(r-s)\lceil n/2\rceil\le t$.
Note that $2r\le t$ implies $\lceil (r(n+1)-2t)/(n-3)\rceil\le r$ 
so $0\le u\le r$. If $u=\lceil (r(n+1)-2t)/(n-3)\rceil$, then clearly
$2u+(r-u)\lceil n/2\rceil\le t$, while if $\lceil (r(n+1)-2t)/(n-3)\rceil<u=0$, then 
from $\lceil (r(n+1)-2t)/(n-3)\rceil<0$ we obtain
$2u+(r-u)\lceil n/2\rceil=r\lceil n/2\rceil\le t$.
\end{proof}

\begin{thm}\label{rhoptsonsmoothconic} Assume the points $p_1,\ldots,p_n$ 
lie on a smooth conic curve $Q'$. 
Let $I=I(Z)$ where $Z=p_1+\cdots+p_n$. Let $m$ and $r$ be positive.
\begin{itemize}
\item[(a)] If $n$ is even or $n=1$, then $I^{(m)}\subseteq I^r$ if and only if $m\ge r$;
in particular, $\rho(I)=1$.
\item[(b)] If $n>1$ is odd, then $I^{(m)}\subseteq I^r$ if and only if $(n+1)r-1\le nm$;
in particular, $\rho(I)=(n+1)/n$.
\end{itemize}
\end{thm}

\begin{proof} (a) For any $n>0$, if $I^{(m)}\subseteq I^r$, then $I^m\subseteq I^{(m)}\subseteq I^r$,
so $r\alpha(I)\le m\alpha(I)$. But $Z\neq0$, so $\alpha(I)>0$, hence $m\ge r$.
Conversely, if $n$ is even or $n=1$, $Z$ is a complete intersection, so $I^r=I^{(r)}$ (see the proof
of Theorem 32 (2), p.\ 110 of \cite{refM}),
hence $I^{(m)}\subseteq I^{(r)}=I^r$ if $m\ge r$.

(b) First say $n=3$. Then $\rho(I)=4/3$ by \cite[Theorem 4.2.3(a)]{refBH},
$\alpha(I)={\rm reg}(I)$ by \cite[Lemma 2.4.2]{refBH},
and $\alpha(I^{(m)})= \lceil 3m/2 \rceil $ by \cite[Lemma 2.4.1 and Example 4.5]{refBH}.
Thus $I^{(m)}\subseteq I^r$ if and only if $2r=r\alpha(I)\le\alpha(I^{(m)})=\lceil 3m/2\rceil$
by Corollary \ref{a=rCor}. But $2r\le \lceil 3m/2\rceil$ if and only if $4r\le 3m+1$.

Now assume $n\ge 5$. Note that if  $(n+1)r-1\le nm$, then $m\ge r$
(for if $m<r$, then $(n+1)r-1< nr$ hence $r< 1$, contrary to hypothesis), and 
we saw above that $I^{(m)}\subseteq I^r$ implies $m\ge r$.
Thus we may assume $m\ge r$. Our result will now follow by Corollary \ref{containmentprop} (using
Lemma \ref{formulaFORq} to compute $q$) once we verify that
$q(I^{(m)}_{\ t}) \ge q(I^r_{\ t})$ for all $t\ge 2m$ if and only if $(n+1)r-1\le nm$.

First assume $m$ is even. Since $n$ is odd, $(n+1)r-1\le nm$
is equivalent to $(n+1)r\le nm$. But if $(n+1)r>nm$, then
$q(I^{(m)}_{\ t}) = 0 < 1 \le q(I^r_{\ t})$ for $t=mn/2$. Conversely,
if $(n+1)r\le nm$, then $(mn-2t)/(n-4)\ge (r(n+1)-2t)/(n-3)$
(and hence $q(I^{(m)}_{\ t}) \ge q(I^r_{\ t})$) for $2m\le t \le nm/2$,
while for $t>mn/2$ we have $q(I^{(m)}_{\ t})= q(I^r_{\ t})=0$ (since both 
$(mn-2t)/(n-4)$ and $(r(n+1)-2t)/(n-3)$ are negative).

Now assume $m$ is odd. If $(n+1)r-1 > nm$, then
$q(I^{(m)}_{\ t}) = 1 < 2 \le q(I^r_{\ t})$ for $t=((m-1)n+4)/2$ (note that this implies $t\ge 2m$, so we 
can apply Lemma \ref{formulaFORq}(a)).
Conversely, if $(n+1)r-1 \le nm$, then
$(mn-2t)/(n-4)\ge (r(n+1)-2t)/(n-3)$ (and hence $q(I^{(m)}_{\ t}) \ge q(I^r_{\ t})$), for $2m\le t\le ((m-1)n+4)/2$
(this is easy to check since we are comparing two linear functions of $t$), while 
$q(I^{(m)}_{\ t}) = 1$ and $q(I^r_{\ t}) \le 1$ for $((m-1)n+4)/2 < t \le (mn-1)/2$,
and $q(I^{(m)}_{\ t}) = q(I^r_{\ t}) = 0$ for $t\ge (mn+1)/2$.
Thus $I^{(m)}\subseteq I^r$ if and only if $(n+1)r-1\le nm$.
In particular, $m/r\ge (n+1)/n>(n+1)/n -1/(rn)$ implies $I^{(m)}\subseteq I^r$
so $\rho(I)\le (n+1)/n$. On the other hand, $m/r < (n+1)/n -1/(rn)$
implies $I^{(m)}\not\subseteq I^r$. Since for any ratio $m/r$ less than
$(n+1)/n$ we can choose $s\gg0$ such that $ms/(rs) < (n+1)/n -1/(rsn)$,
we have $I^{(ms)}\not\subseteq I^{rs}$ so $m/r \le \rho(I)$ and thus
$(n+1)/n \le \rho(I)$; i.e., $\rho(I)=(n+1)/n$.
\end{proof}

\section{General Points}\label{genpts}

In this section we determine $\rho(I)$ and solve the containment problem
for each set of $n\le9$ general points of $\pr2$. Any $n\le5$ general
points lie on a smooth conic and hence these cases have been dealt with in 
the previous section, so now we consider the ideal $I$ of $n$ general
points for $6\le n\le 9$. In this section, $X$ will denote
the blow up of $\pr2$ at these $n$ points.
(For each $n$, once we determine exactly when 
$I^{(m)}\subseteq I^r$ occurs, our determination
of $\rho(I)$ uses the same argument as used at the end of 
the proof of Theorem \ref{rhoptsonsmoothconic}(b),
so we merely state the value of $\rho(I)$ without repeating the justification.)

\begin{prop} Let $I$ be the ideal of $n=6$ general points of $\pr2$.
Then $I^{(m)}\subseteq I^r$ if and only if 
$m\ge \frac{5}{4}r-\frac{5}{12}$, and thus $\rho(I)=\frac{5}{4}$.
\end{prop}

\begin{proof} It is easy to see that $\alpha(I)={\rm reg}(I)=3$.
Now note that $25L-10E$ is nef, where $E=E_1+\cdots+E_6$,
since $25L-10E=L+2\sum_i(2L-E+E_i)$ is the class of a curve with 7 irreducible components
(coming from the line and the 6 conics through each subset of 5 of the 6 points)
and meets each component non-negatively.
Thus $5(5L-2E)$ is nef, hence
$(\alpha L-mE)\cdot (5L-2E)\ge0$ for $\alpha=\alpha(I^{(m)})$. 
Thus $5\alpha-12m\ge0$ or $\alpha\ge \frac{12m}{5}$.
But $\lceil\frac{12m}{5}\rceil L-mE$ is a non-negative integer linear combination
of the classes $3L-E$, $5L-2E$ and $12L-5E$ (just check mod 5), each of which is 
(by counting constants) effective.
Thus $\frac{12m}{5}\le\alpha(I^{(m)})\le\lceil\frac{12m}{5}\rceil$, so $\alpha(I^{(m)})=\lceil\frac{12m}{5}\rceil$
and $\gamma(I)=\frac{12}{5}$. 
Next, note that $m\ge \frac{5}{4}r-\frac{5}{12}$ if and only if
$\frac{12m}{5}\ge3r-1$ if and only if $\frac{12m}{5}>3r-1$ if and only if 
$\alpha(I^{(m)})=\lceil\frac{12m}{5}\rceil\ge 3r=
r{\rm reg}(I)$. Now apply Corollary \ref{a=rCor}.
\end{proof}

\begin{Rmk} The behavior of powers of an ideal $I$ of points is especially simple 
when, as is the case for $n=6$ general points in $\pr2$,
$\alpha(I)={\rm reg}(I)$: the powers are obtained by truncating homogeneous 
components of $I^{(r)}$ of low degree. 
For example, assume that $I\subsetneq k[\pr N]$ is a radical ideal
defining a finite nonempty set of points in $\pr N$ such that $\alpha(I)={\rm reg}(I)$.
Then $I^r= I^{(r)}\cap M^{r\alpha}$, where $M=(x_0,\ldots,x_N)\subset k[\pr N]$ is 
the irrelevant ideal and $\alpha=\alpha(I)$: clearly $I^r \subseteq I^{(r)}\cap M^{r\alpha}\subseteq I^{(r)}$ so
$I^r_{\ t} \subseteq(I^{(r)}\cap M^{r\alpha})_{t}=0$ for $t<r\alpha$,
while $I^r_{\ t}=I^{(r)}_{\ t}$ for $t\ge r\alpha=r{\rm reg(I)}$ by \cite[Lemma 2.3.3(c)]{refBH}.
\end{Rmk}

\begin{prop}
Let $I$ be the ideal of $n=7$ general points of $\pr2$.
Then $I^{(m)}\subseteq I^r$ if and only if either $m\ge \frac{8}{7}r$ or $r=m=1$, 
and thus $\rho(I)=\frac{8}{7}$.
\end{prop}

\begin{proof} Let $E=E_1+\cdots+E_7$ and let $C$ be a general (hence smooth
and irreducible) curve $C\in |-K_X|$; thus $[C]=[3L-E]$.

First, note for each $i$ that $C-E_i$ is linearly equivalent to an effective divisor $C_i$ (by counting constants) 
which is reduced and irreducible (otherwise $C-E_i$ would be a sum of two or more
effective divisors, at least one of which would come from either
a line through 3 or more points or a conic through 6 or more points, but neither can happen
since the points are general). Thus $3(8L-3E)=C+\sum_iC_i$ is nef
since it meets each component non-negatively, hence
$(\alpha L-mE)\cdot (8L-3E)\ge0$ for $\alpha=\alpha(I^{(m)})$. 
Thus $8\alpha-21m\ge0$ so $\alpha\ge \lceil\frac{21m}{8}\rceil$.
But $\lceil\frac{21m}{8}\rceil L-mE$ is a non-negative integer linear combination
of the classes $3L-E$, $8L-3E$ and $21L-8E$ (just check mod 8), each of which is effective.
Thus $\alpha\le \lceil\frac{21m}{8}\rceil$, so $\alpha=\lceil\frac{21m}{8}\rceil$. 

We claim that $I^r_{\ t}=I^{(r)}_{\ t}$ for all $t\ge 3r+1$. To justify this, it is enough to show that
$(I_3)^{r-1}(I_{t-3r+3})=I^{(r)}_{\ t}$; i.e., that 
$H^0(X, -K_X)^{\otimes (r-1)}\otimes H^0(X, iL-K_X)\to H^0(X, iL-rK_X)$
is onto for each $i>0$ and each $r\ge1$.
To prove this, it is enough in fact to show that 
$H^0(X, -K_X)\otimes H^0(X, iL-mK_X)\to H^0(X, iL-(m+1)K_X)$
is onto for each $i>0$ and each $m\ge0$.
Now consider the following diagram, where $V_F=H^0(X, F)$
and the vertical maps are the canonical multiplication maps:

$$\begin{matrix}
0 & \to & H^0(X, G-C)\otimes V_F  & \to & H^0(X,G)\otimes V_F &
\to & H^0(C, G\vert_C)\otimes V_F & \to & 0 \cr
{} & {} & \downarrow \raise3pt\hbox to0in{$\scriptstyle\mu_1$\hss} & {}
& \downarrow\raise3pt\hbox to0in{$\scriptstyle\mu_2$\hss} & {}
& \downarrow\raise3pt\hbox to0in{$\scriptstyle\mu _3$\hss} &
{} & {} \cr
0 & \to & H^0(X,F+G-C) & \to & H^0(X,F+G) & \to &
H^0(C,(F+G)\vert_C)& \to & 0 \cr
\end{matrix}$$

Consider the case that $F=-K_X$ and $G=iL-mK_X$.
To see that the rows of the diagram are exact, note that since $C$ is a prime divisor of 
non-negative self-intersection, it is nef, hence so are $-mK_X$ and $iL-mK_X$ for any $i>0$ and $m\ge0$,
hence $h^1(X, -mK_X)=0$ and $h^1(X,iL-mK_X)=0$ by \cite{refH5}.
Note when $m=1$ that $\mu_1$ 
(i.e., $H^0(X, iL)\otimes H^0(X, -K_X)\to H^0(X, iL-K_X)$) is onto by \cite{refH3}. 
If we show that $\mu_3$ is onto for each $m\ge1$, then the snake lemma applied to the diagram and
induction on $m$ show that $\mu_2$ is onto for all $m\ge1$. (We can at least see that
$\mu_3$ is onto for $m=0$: apply the snake lemma to the diagram above
with $F=iL$ and $G=-K_X$, using the fact that in this case $\mu_2$ is 
$H^0(X, -K_X)\otimes H^0(X, iL)\to H^0(X, iL-K_X)$, which we just noted is onto.)

To show that $\mu_3$ is onto, let $Z\in|C|_C|$ be a divisor on $C$
consisting of two distinct points and consider the diagram
$$\begin{matrix}
0 & \to & H^0(C, G-C)\otimes V_F  & \to & H^0(C,G)\otimes V_F &
\to & H^0(Z, G\vert_Z)\otimes V_F & \to & 0 \cr
{} & {} & \downarrow \raise3pt\hbox to0in{$\scriptstyle\mu_1'$\hss} & {}
& \downarrow\raise3pt\hbox to0in{$\scriptstyle\mu_2'$\hss} & {}
& \downarrow\raise3pt\hbox to0in{$\scriptstyle\mu _3'$\hss} &
{} & {} \cr
0 & \to & H^0(C,F+G-C) & \to & H^0(C,F+G) & \to &
H^0(Z,(F+G)\vert_Z)& \to & 0 \cr
\end{matrix}$$
where now $F=-K_X|_C$, $V_F=H^0(C, F)$ and $G=(iL-mK_X)|_C$. 
The map $\mu_1'$ is onto for $m=1$ by the parenthetical remark
at the end of the preceding paragraph. Thus the snake lemma applied to the
new diagram, using induction on $m$, implies $\mu_2'$ is onto for all $m\ge1$
if we show $\mu_3'$ is onto. Since
$0\to \OO_X\to \OO_X(-K_X)\to \OO_C(-K_X)\to 0$ is exact on global sections
the map $H^0(X, F)\to H^0(C, F|_C)$ is surjective, 
hence $\mu_2'$ and $\mu_3$ have the same image. Thus surjectivity of $\mu_2'$ implies that of
$\mu_3$, which is what we want to show. To see that $\mu_3'$ is onto, note that
$H^0(Z,(F+G)\vert_Z)$ is just the vector space of $k$-valued functions on the two points
of $Z$, as is $H^0(Z, G\vert_Z)$, and $\mu_3'$ just multiplies a function
in $H^0(Z, G\vert_Z)$ by the values of an element of $V_F$ at the two points.
But $|-K_X|_C|$ is base point free (since a complete linear of degree 2 on an elliptic curve
$C$ is base point free); i.e.,
$V_F=H^0(C, F)$ is base point free, hence there is an element of $V_F$ which is nonzero at both points,
so $\mu_3'$ is onto, which finishes our justification that $I^r_{\ t}=I^{(r)}_{\ t}$ for all $t\ge 3r+1$.

We now proceed to find all $m>0$ for each $r>0$ such that $I^{(m)}\subseteq I^r$.
If $\lceil\frac{21m}{8}\rceil < 3r$, then $\alpha(I^{(m)})<\alpha(I^r)$, so $I^{(m)}\not\subseteq I^r$.
If $\lceil\frac{21m}{8}\rceil \ge 3r$, then $m\ge r$
(if not, then $r>m$, but $\lceil\frac{21m}{8}\rceil \ge 3r$ implies $\frac{21m}{8} > 3r-1$
and so $\frac{8}{3}>8r-7m$ which with $r>m$ gives $\frac{8}{3}>8r-7m>r$, so $r=1$ and $m=0$,
contrary to hypothesis), so $I^{(m)}\subseteq I^{(r)}$.
Now suppose $\lceil\frac{21m}{8}\rceil > 3r$; then $I^{(m)}\subseteq I^r$, since $0=I^{(m)}_{\ t}\subseteq I^r_{\ t}$
for $t<\alpha(I^{(m)})=\lceil\frac{21m}{8}\rceil$, and $t\ge\alpha(I^{(m)})=\lceil\frac{21m}{8}\rceil$
implies $t\ge 3r+1$, so $I^{(m)}_{\ t}\subseteq I^{(r)}_{\ t}=I^r_{\ t}$. 

So suppose $\alpha=\lceil\frac{21m}{8}\rceil = 3r$. 
By the same arguments, $I^{(m)}_{\ t}\subseteq I^r_{\ t}$ if either $t<3r$ or $t>3r$,
so $I^{(m)}\subseteq I^r$ if and only if $I^{(m)}_{\ 3r}\subseteq I^r_{\ 3r}$.
By checking mod 8,
we can see $\lceil\frac{21m}{8}\rceil = 3r$ implies $m$ is congruent to 0, 1 or 2 modulo 8.
If $m\equiv0$ mod 8, then $m=8i$ for some $i$, so $r=7i$ and $\alpha=21i$.
But $|\alpha L - mE|=|i\sum_jC_j|=i\sum_jC_j$ is fixed,
since the curves $C_j$ are disjoint prime divisors of negative self-intersection.
Thus $I^{(m)}_{\ \alpha}$ is 1-dimensional, spanned by the product of the $r=7i$ cubic forms
corresponding to the summands of $i\sum_jC_j$, with each cubic form
vanishing at each of the $n=7$ points and hence being in $I_3$. 
Thus $I^{(m)}_{\ \alpha}\subseteq I^r_{\ \alpha}$. 

Say $m\equiv1$ mod 8, so $m=8i+1$ and $r=7i+1$
for some $i$. If $m=1$, then clearly $I=I^{(m)}\subseteq I^r$ if and only if $r=1$,
so say $m>1$ and hence $i\ge 1$. Then $I^{(m)}_{\ 3r}=H^0(X, 3(8L-3E)+(i-1)\sum_jC_j)$; 
note that $|3(8L-3E)+(i-1)\sum_jC_j|=
|3(8L-3E)|+(i-1)\sum_jC_j$ since the $C_j$ are disjoint and of negative self-intersection
with $(8L-3E)\cdot C_j=0$, so $(i-1)\sum_jC_j$ is fixed in the linear system. 
Let $C$ be a smooth section of $|-K_X|$.
By Serre duality, $h^1(X, \OO_X(3(8L-3E)-C))=h^1(X, \OO_X(-(3(8L-3E))))$ and
since $3(8L-3E)$ is nef and big, $h^1(X, \OO_X(-(3(8L-3E))))=0$ by 
Ramanujam vanishing (for the characteristic $p$ version,
see Theorem 1.6 \cite{refT} or Theorem 2.8 of \cite{refH1}).
Thus $0\to \OO_X(3(8L-3E)-C)\to \OO_X(3(8L-3E))\to \OO_C(3(8L-3E))\to 0$
is exact on global sections, and $\OO_C(3(8L-3E))$ is very ample
(since it has degree 9, and any divisor of degree at least 3 on an elliptic curve is very ample).
In particular, the trace of $|3(8L-3E))|$ on $C$ is not composed with a pencil.
Since the trace of $I^r_{\ 3r}=(I_3)^r$ on $C$ is composed with the pencil given by the trace of
$|-K_X|$ on $C$, we see
that $I^{(m)}_{\ 3r}\not\subseteq I^r_{\ 3r}$, and hence $I^{(m)}\not\subseteq I^r$.

Finally, say $m\equiv2$ mod 8 (i.e., $m=8i+2$ and $r=7i+2$
for some $i$). Then $I^{(m)}_{\ 3r}=H^0(X, 6(8L-3E)+(i-2)\sum_iC_i)$ for $i\ge2$,
while $I^{(m)}_{\ 3r}=H^0(X, 3(8L-3E)+C)$ for $i=1$ and $I^{(m)}_{\ 3r}=H^0(X, 2C)$ for $i=0$. 
Arguing as before, the trace of $I^{(m)}_{\ 3r}$ on $C$ is not composed with a pencil but
the trace of $I^r_{\ 3r}$ on $C$ is composed with the same pencil as before. Thus
$I^{(m)}_{\ 3r}\not\subseteq I^r_{\ 3r}$, so $I^{(m)}\not\subseteq I^r$.

Reviewing, we have $I^{(m)}\subseteq I^r$ if and only if $m=r=1$, or 
$\lceil\frac{21m}{8}\rceil > 3r$ (which is equivalent to $\frac{21m}{8} > 3r$; i.e., to $m/r>8/7$), 
or $m=8i$ and $r=7i$ for $i\ge 1$ (which is equivalent to $m/r=8/7$).
Thus $I^{(m)}\subseteq I^r$ if and only if $m=r=1$ or $m/r\ge 8/7$.
\end{proof}

The following result is interesting in that early evidence suggested that
$\rho(I)\le \sqrt{2}$ for $n$ generic points of $\pr2$ (see, for example,
\cite[Corollary 1.3.1, Proposition 4.4]{refBH}). For $n=8$, however, we now see that
$\rho(I)>\sqrt{2}$ (but only by a bit).

\begin{prop}
Let $I$ be the ideal of $n=8$ general points of $\pr2$.
Then $I^{(m)}\subseteq I^r$ if and only if either $m=r=1$ or
$m\ge \frac{17}{12}r-\frac{1}{3}$, and thus $\rho(I)=\frac{17}{12}$. 
\end{prop}

\begin{proof}
Let $C'$ be a general cubic through the 8 general points $p_i$.
Thus we may assume that $C'$ is smooth and the class of its proper transform $C$
is $-K_X$. Since $C^2=1\ge0$, we see that $-K_X$ is nef.

Now $\alpha=\alpha(I^{(m)})=\lceil\frac{48m}{17}\rceil$: let $F=(\lceil\frac{48m}{17}\rceil)L-mE$,
where $E=E_1+\cdots+E_8$. Then $F$ is a non-negative integer linear combination of $-K_X$, 
$17L-6E$ and $48L-17E$, and these all (by counting constants) are 
effective. Thus $\alpha\le\lceil\frac{48m}{17}\rceil$.
But $17L-6E$ is nef (since it reduces by quadratic Cremona transformations to $L$ \cite{refH6}), 
so $(\alpha L-mE)\cdot(17L-6E)\ge0$, hence $\alpha\ge \lceil\frac{48m}{17}\rceil$.

Next, $\alpha\ge 4r$ if and only if $\frac{48m}{17}>4r-1$ if and only if 
$48m>68r-17$ if and only if $48m\ge68r-16$ if and only if $m\ge\frac{17}{12}r-\frac{1}{3}$. 
Note also that ${\rm reg}(I)=4$.
Thus, if $m\ge\frac{17}{12}r-\frac{1}{3}$, then $\alpha(I^{(m)})\ge 4r=r\,{\rm reg}(I)$ so
(by \cite[Lemma 2.3.4]{refBH}) $I^{(m)}\subseteq I^r$. Also, if $m=r=1$, then
$I^{(m)}=I=I^r$.

Conversely, assume $m<\frac{17}{12}r-\frac{1}{3}$, hence $\alpha(I^{(m)})< 4r$. If $m=1$ and $r>1$,
then $\alpha(I^{(m)})=3<3r=\alpha(I^r)$, so $I^{(m)}\not\subseteq I^r$. So assume $m>1$ and $t=\alpha(I^{(m)})$; 
we will show that $I^r_{\ t}$ has a base point which $I^{(m)}_{\ t}$ doesn't have. It thus follows
that $I^{(m)}\not\subseteq I^r$.

If $t=\alpha(I^{(m)})<r\alpha(I)=3r$, this is clear, since $I^{(m)}_{\ t}\ne0$ but $I^r_{\ t}=0$. So assume
$r\alpha(I)\le t=\alpha(I^{(m)})<4r$; i.e., assume $3r-1 < 48m/17 \le 4r-1$. Note that
$I^r_{\ t}=\sum_{t_1\le \cdots\le t_r}\Pi_iI_{t_i}$, where the sum is over all sequences $t_i$
with $3\le t_1\le \cdots\le t_r$ such that $\sum_it_i=t$ (we ignore sequences with $t_1<3$ since $I_j=0$ for $j<3$). 
Since $t<4r$, we see $t_1=3$ for every sequence, hence $\Pi_iI_{t_i}$ and thus
$I^r_{\ t}$ has the same base points as does $I_3$, these being the nine base points $p_1,\cdots,p_9$
of the pencil of cubics $I_3$ through our $n=8$ general points $p_1,\cdots,p_8$. So now it suffices to show
that $p_9$ is not a base point of $I^{(m)}_{\ t}$. 

First we check that the points $p_1,\cdots,p_9$ are distinct. Note that $p_9$ can be identified with the base point
of $|C|$ (i.e., under the morphism $\pi:X\to\pr2$ blowing up the 8 points, if $p$ is the base point of $|C|$, then $\pi(p)=p_9$).
Let $q_i$ be the point where $E_i$ meets $C$; then $\pi(q_i)=p_i$. Thus to show that $p_i\ne p_9$ for $i<9$, 
it is enough to show that $q_i\ne p$; i.e., that the restriction of $C-E_i$ to $C$ is not trivial, which is the same as 
showing that $h^0(C, \OO_C(C-E_i))=0$, since a line bundle of degree 0 on an elliptic curve has positive $h^0$
if and only if the line bundle is trivial (in which case $h^0=1$). So suppose that $h^0(C, \OO_C(C-E_i))=1$.
Note $h^0(X,\OO_X(-E_i))=0$ (since $L-E_i$ is nef and $(L-E_i)\cdot (-E_i)<0$), $h^2(X, \OO_X(-E_i))=h^0(X,\OO_X(K_X+E_i))=0$
(by duality and the facts that $L\cdot(K_X+E_i)<0$ and that $L$ is nef) and $h^1(X, \OO_X(-E_i))=0$
(by Riemann-Roch since $h^0(X,\OO_X(-E_i))-h^1(X,\OO_X(-E_i))+h^2(X,\OO_X(-E_i))=((-E_i)^2-K_X\cdot(-E_i))/2+1 = 0$).
Taking cohomology of the exact sequence $0\to \OO_X(-E_i)\to \OO_X(C-E_i)\to \OO_C(C-E_i)\to 0$,
we now see that $h^0(X, \OO_X(C-E_i))=h^0(C, \OO_C(C-E_i))=1$. For $i=1$ (just to be specific; 
the argument for $1<i\le8$ is the same) we have $C-E_1=3L-2E_1-E_2-\cdots - E_8$. But 
$h^0(X, \OO_X(3L-2E_1))=\dim I(2p_1)_3=7$, and for a general point $p_{j+1}$
with $j\ge 1$, the dimension of $|3L-2E_1-E_2-\cdots - E_{j+1}|$ is one less than the dimension of 
$|3L-2E_1-E_2-\cdots - E_j|$ as long as $|3L-2E_1-E_2-\cdots - E_j|$ is not empty (since assigning 
a general base point to any nonempty linear system imposes exactly one condition). Thus
$h^0(X,\OO_X(C-E_i))$ is 0 for general points, contradicting $h^0(C, \OO_C(C-E_i))=1$. 

It is well known that the divisor class $[2C-E_i]=[6L-2(E_1+\cdots+e_8)-E_i]$ is the class of a smooth irreducible curve $D_i$
when the points $p_1,\cdots,p_8$ are general; indeed, in this situation, $[2C-E_i]$ reduces by quadratic Cremona
transformations to $[E_1]$ (see \cite{refH6}, for example), hence $D_i$ is a smooth rational curve of self-intersection 
$-1$ which meets $C$ at a single point $q_i'$, since $C\cdot D_i=1$. As before, $p\ne q_i'$. (If $p=q_i'$, then
$\OO_C(D_i-C)$, which is equal to $\OO_C(C-E_i)$, would be trivial, and we have seen it is not.)

Under the identification of $I^{(m)}_{\ t}$ with $H^0(X,F)$ for $F=tL-mE$, if $|F|$
has a base point $p'$, then $p''=\pi(p')$ is a base point of $I^{(m)}_{\ t}$. Not every base point of
$I^{(m)}_{\ t}$ comes from a base point of $|F|$, since $p_1,\ldots,p_8$ are always base points
of $I^{(m)}_{\ t}$ even though $|F|$ can sometimes be base point free (after all, one of the motivations historically 
for blowing up points was to remove base points). However, if $I^{(m)}_{\ t}$ has a base
point $p''$ away from the points $p_1,\ldots,p_8$ blown up by $\pi$, then $|F|$ has a base point $p'$ with $\pi(p')=p''$.
Thus, since the points $p_1,\cdots,p_9$ are distinct, $p_9$ is a base point of $I^{(m)}_{\ t}$ if and only if
$p$ is a base point of $|F|$. But, as we saw above,
$F=aC+b(17L-6E)+c(48L-17E)$ for some non-negative integers $a, b$ and $c$. Since we are assuming that
$m>1$, either $a>1$, or $b$ or $c$ is positive. If $b>0$, then $aC+b(17L-6E)$ is nef with $(aC+b(17L-6E))\cdot C>1$,
so $|aC+b(17L-6E)|$ is base point free by \cite{refH1}. Also $48L-17E=D_1+\cdots+D_8$, so $p$ is not a base point
of $|48L-17E|$ and hence not of $|F|$ when $b>0$ (since then
we can write $F$ as a sum of effective divisors, none of which pass through $p$).
So suppose $b=0$. If $a>1$, then $|aC|$ is nef with $aC\cdot C>1$ hence 
base point free by \cite{refH1}, so the same argument again shows
$p$ is not a base point of $|F|$. If $a=1$ but $c>0$, we reduce to the case
that $b>0$, since $F=(C+(48L-17E))+(c-1)(48L-17E)=3(17L-6E)+(c-1)(48L-17E)$. 
\end{proof}

In stating the next result one needs to be careful.
The issue is that the condition of generality on the points which guarantees
that the result holds depends on $m$. The reason is that for 9 points,
$h^0(X,-mK_X)=1$ holds for general points, but the open condition for which
this holds becomes smaller as $m$ increases.

\begin{prop} Fix positive integers $m$ and $r$.
Then $I^{(m)}\subseteq I^r$ where $I$ is the ideal of $n=9$ general points of $\pr2$
if and only if $m\ge \frac{4}{3}r-\frac{1}{3}$. For $n=9$ generic points we thus have $\rho(I)=\frac{4}{3}$.
\end{prop}

\begin{proof}
Let $F=\alpha L-mE$ for $E=E_1+\cdots+E_9$ and $\alpha=\alpha(I^{(m)})$. 
Let $C'$ be a general cubic curve through $p_1,\ldots,p_9$
and let $C$ be its proper transform. Then $C$ is smooth and irreducible with $C^2=0$, hence nef.
Thus $F\cdot C\ge0$ implies $\alpha\ge 3m$, but $3mL-mE=mC$, so $\alpha\le 3m$.
Thus $\alpha(I^{(m)})=3m$ but, for $t\ge 3m$, $I^{(m)}_{\ t}$ is fixed component
free (by \cite{refH6}) if and only if $t\ge 3m+1$, whereas for $t\ge \alpha(I^r)=3r$, 
$I^r_{\ t}$ is fixed component free if and only if there is a product $I_{t_1}\cdots I_{t_r}$ for
a sequence $t_1\le \cdots\le t_r$ with $\sum_it_i=t$ with $t_1>3$;
i.e., if and only if  $t\ge 4m$. Since $L-K_X$ is normally generated \cite{refH1} 
and $|iL-K_X|$ is fixed component free \cite[Proposition 3.2.1.1(b)]{refH2}
and $H^0(X, L-K_X)\otimes H^0(X, iL)\to H^0(X, (i+1)L-K_X)$ is onto \cite[Theorem 3.2.1.2(b)]{refH2}
for $i>0$, we have $I_{t_1}\cdots I_{t_r}=I^{(r)}_{\ t}$ for all sequences $4\le t_1\le \cdots\le t_r$
with $\sum_it_i=t\ge 4r$ and hence $I^r_{\ t}=I^{(r)}_{\ t}$. 
Thus $\alpha(I^{(m)})\le4r-2$ implies $I^{(m)}\not\subseteq I^r$ (since there is a $t$
with $\alpha(I^{(m)})=3m< t<4r$, so $I^{(m)}_{\ t}$ has no fixed components but $I^r_{\ t}$ does, 
hence $I^{(m)}_{\ t}\not\subseteq I^r_{\ t}$). If $\alpha(I^{(m)})\ge4r$, then 
$I^{(m)}\subseteq I^{(r)}$ by \cite[Lemma 2.3.4]{refBH} since ${\rm reg}(I)=4$. Finally,
if $3m=\alpha(I^{(m)})=4r-1$, then $I^{(m)}\subseteq I^r$, since
$I^{(m)}_{\ t}\subseteq I^r_{\ t}$ for all $t\ge0$. To see this, note first that $3m=4r-1$ implies
$m\ge r$, so $I^{(m)}\subseteq I^{(r)}$. Also note that $I^{(r)}_{\ t}=I^r_{\ t}$ for
$t\ge 4r$ by \cite[Lemma 2.3.3(c)]{refBH} and that $I^{(m)}_{\ 3m}=I^{m}_{\ 3m}$,
since $h^0(X, -mK_X)=1$, so $|-mK_X|=mC$. Now we see that $I^{(m)}_{\ t}\subseteq I^r_{\ t}$
for $t<\alpha(I^{(m)})$ (since $I^{(m)}_{\ t}=0$), 
for $t=\alpha(I^{(m)})$ (since $I^{(m)}_{\ 3m}=I^{m}_{\ 3m}$ and $m\ge r$ so $I^{m}_{\ 3m}\subseteq I^r_{\ 4r-1}$);
and for $t>\alpha(I^{(m)})$ (since $I^{(m)}_{\ t}\subseteq I^{(r)}_{\ t}=I^r_{\ t}$).
Thus $I^{(m)}\subseteq I^r$ if and only if $3m=\alpha(I^{(m)})\ge 4r-1$, if and only if $m\ge \frac{4}{3}r-\frac{1}{3}$. 
\end{proof}

\section{Examples and Questions}\label{quests}

It is an interesting problem to determine for which ideals we have $I^{(r)}=I^r$ for all $r\ge1$.
It follows by Macaulay's Umixedness Theorem that $I^{(r)}=I^r$ holds for complete intersections
(see the proof of Theorem 32 (2), p. 110 of \cite{refM}).
More recently, characterizations of ideals for which symbolic and ordinary powers coincide
have been given by \cite{refHo} (for prime ideals) and by \cite{refLS} 
(for radical ideals). Such ideals have been studied also in \cite{refMNV}, \cite{refMo} and \cite{refHH}.
In the following example we give constructions of ideals for which
$I^{(r)}=I^r$ for all $r\ge 1$ which seem to be new.

\begin{Ex} Consider a fat point subscheme coming from the class
of a smooth rational curve of self-intersection 0 on the blow up $X$
of $\pr 2$ at general points $p_1,\cdots,p_n\in\pr 2$. 
There are many additional examples which can be obtained from the
one below by using the action of the Cremona group. For specificity,
let $d$ be a positive integer bigger than 2, let $n=2d$,
and let $p_1,\cdots,p_n\in\pr 2$ be general points. Let $Z=(d-1)p_1+p_2+\cdots+p_n$.
Then $F=dL-(d-1)E_1-E_2-\cdots-E_n$ is (linearly equivalent to) an effective divisor
by Riemann-Roch; the general member of $|F|$ is a smooth rational curve, the proper transform 
$D$ in fact of an irreducible degree $d$ curve $C'$ with a singularity of
multiplicity $d-1$ (the curve $C'$ can be obtained by applying quadratic Cremona
transformations to a line in $\pr 2$). It follows that $\alpha(I(Z))\le d$ 
but $F$ is nef with $F^2=0$ so $\alpha(I(Z))$ cannot be less than $d$
(since $\alpha(I(Z))<d$ implies $F-L$ is effective, but
$F\cdot (F-L)=-d<0$, which is impossible since $F$ is nef). 

Consider $0\to \mathcal{O}_X((m-1)F)\to \mathcal{O}_X(mF)\to \mathcal{O}_D(mF)\to 0$.
The restriction $(mF)\vert_D$ is trivial, since $D$ is smooth and rational
and $F^2=0$. 
By induction on $m$, taking cohomology of
$0\to \mathcal{O}_X((m-1)F)\to \mathcal{O}_X(mF)\to \mathcal{O}_D\to 0$
and using the fact that $h^1(X, \mathcal{O}_X)=0$, we see
that $h^1(X, \mathcal{O}_X(mF))=0$ for all $m\ge0$. 
It now follows by induction that $h^0(X, \mathcal{O}_X(mF))=m+1$. A similar
argument applied to 
$0\to \mathcal{O}_X((m-1)F+L)\to \mathcal{O}_X(mF+L)\to \mathcal{O}_D(L)\to 0$
gives $h^1(X, \mathcal{O}_X(mF+L))=0$
and $h^0(X, \mathcal{O}_X(mF+L))=m(d+1)+3$.

Consider the following diagram, where $V_L=H^0(X, L)$ and 
the vertical maps are the canonical multiplication maps:
$$\begin{matrix}
0 & \to & H^0(X, (m-1)F)\otimes V_L  & \to & H^0(X,mF)\otimes V_L &
\to & H^0(D, (mF)\vert_D)\otimes V_L & \to & 0 \cr
{} & {} & \downarrow \raise3pt\hbox to0in{$\scriptstyle\mu_1$\hss} & {}
& \downarrow\raise3pt\hbox to0in{$\scriptstyle\mu_2$\hss} & {}
& \downarrow\raise3pt\hbox to0in{$\scriptstyle\mu _3$\hss} &
{} & {} \cr
0 & \to & H^0(X,(m-1)F+L) & \to & H^0(X,mF+L) & \to &
H^0(D,(mF+L)\vert_D)& \to & 0 \cr
\end{matrix}$$

By induction on $m$ using the snake lemma and the fact that $d>2$, 
the maps $\mu_1, \mu_2$ and $\mu_3$ are injective, 
hence we have exact sequences
$0\to {\rm cok}(\mu_1)\to {\rm cok}(\mu_2)\to {\rm cok}(\mu_3)\to 0$
and $0\to {\rm Im}(\mu_1)\to {\rm Im}(\mu_2)\to {\rm Im}(\mu_3)\to 0$.
Since $h^0(X, \mathcal{O}_X(F))=2$, there is a section $C\in |F|$ disjoint
from $D$. Taking $m=1$, let $C_1,\ldots,C_{d-2}$ be sections of $|F+L|$
which span a subspace complementary to the image of $\mu_2$.
Let $L_1, L_2$ and $L_3$ be a basis for $V_L$.
By induction we find for every $m$, that 
the image of $\mu_2$ is spanned by $iC+jD+L_k$
for $0\le i,j\le m$, $i+j=m$ and $1\le k\le 3$, and
$H^0(X,mF+L)$ is spanned by these together with
$i'C+j'D+C_l$, $0\le i',j'\le m-1$, $i'+j'=m-1$, $1\le l\le d-2$. 
(For the induction, note that the basis for $H^0(X,(m+1)F+L)$
comes partly from the basis for $H^0(X,mF+L)$, by adding
$D$ to the basis elements of $H^0(X,mF+L)$, and partly
from $H^0(C,((m+1)F+L)\vert_D)\cong H^0(C,L\vert_D)$,
where the isomorphism takes the restriction of $mC+C_i$
to that of $C_i$.)

Expressing the foregoing in terms of ideals,
we have that the regularity of $I(mZ)$ is at most $m\alpha+1$
and thus that $I(mZ)$ is generated in degrees $m\alpha$
and $m\alpha+1$ and, as long as $d>2$, both degrees are needed
(where we denote $\alpha(I(Z))$ simply by $\alpha$).
Elements $A$ and $B$, corresponding to $C$ and $D$ above, span $I(Z)_\alpha$, and 
$C'_1,\ldots,C'_{d-2}$, corresponding to the $C_i$ above, 
span a subspace of $I(Z)_{\alpha+1}$ complementary 
to the image of $I(Z)_{\alpha}\otimes R_1\to I(Z)_{\alpha+1}$,
where $R=k[\pr 2]$ so $R_1$ denotes the vector space
span of the linear forms. Then
$A^iB^{m-i}$ for $0\le i\le m$ span $I(mZ)_\alpha$, and
$A^iB^{m-i-1}C'_j$ for $0\le i\le m-1$ and $1\le j\le d-2$
span a subspace of $I(mZ)_{m\alpha+1}$ complementary 
to the image of $I(Z)^m_{\ m\alpha}\otimes R_1\to I(Z)_{m\alpha+1}$.
Thus the elements $A^iB^{m-i}$ and $A^iB^{m-i-1}C'_j$ generate the ideal
$I(mZ)$, but all of these elements are in $I(Z)^m$, hence
$I(Z)^m\subseteq I(mZ)\subseteq I(Z)^m$; i.e., $I(Z)^m=I(mZ)=I(Z)^{(m)}$,
and clearly there is no $m\ge1$ such that $I(mZ)$ is a power of any ideal which is
prime (since the support of $mZ$ consists of a finite set with more than one point), radical 
(since the points do not all have the same multiplicity) or a complete intersection
(since if $I(mZ)=J^r$ for some complete intersection $J$,
then $I(msZ)=J^{rs}$ for all $s\ge1$ and hence $J^{sr}$ is generated either in 
one degree or in a range of degrees that increases with $s$, but
$I(msZ)$ is generated in two degrees for all $s$). 
\end{Ex}

If $I$ is homogeneous with $0\ne I\subsetneq k[\pr N]$ and 
$I^{(r)}=I^r$ for all $r\ge 1$, then $\rho(I)=1$.
However, we do not know any examples with $\rho(I)=1$ 
but for which $I^{(r)}=I^r$ fails for some $r$. This raises the 
following question:

\begin{Ques} Let $0\ne I\subsetneq k[\pr N]$ be a homogeneous ideal.
Does $\rho(I)=1$ imply $I^{(r)}=I^r$ for all $r\ge 1$?
\end{Ques}


\begin{thebibliography}{FHH}

\bibitem[PSC]{refPSC} T.\ Bauer, S.\ Di Rocco, B.\ Harbourne, M.\ Kapustka, A.\ Knutsen, W.\
Syzdek, and T.\ Szemberg. {\it A primer on Seshadri constants}, 
to appear in the AMS Contemporary Mathematics series volume 
``Interactions of Classical and Numerical Algebraic Geometry,"
Proceedings of a conference in honor of A.J. Sommese, held at Notre Dame, May 22--24 2008.

\bibitem[BH]{refBH} C.\ Bocci and B.\ Harbourne. {\it Comparing 
powers and symbolic powers of ideals}, to appear, Journal of Algebraic Geometry.

\bibitem[C]{refC} G.\ Castelnuovo. {\it Ricerche generali sopra i sistemi lineari di curve piane}, Mem. 
Accad. Sci. Torino, II 42 (1891).

\bibitem[CM]{refCM} C.\ Ciliberto and R.\ Miranda. {\it NagataÕs Conjecture for a Square Number of Points}, 
Ricerche di Matematica, 55 (2006), 71--78. 

\bibitem[ELS]{refELS} L. Ein, R. Lazarsfeld and K. Smith. {\it Uniform bounds and 
symbolic powers on smooth varieties}, Invent. Math. 144 (2001), p. 241-252.

\bibitem[E]{refE} L.\ Evain. {\it Computing limit linear series with infinitesimal methods},
preprint (ArXiv math.AG/0407143).


\bibitem[G]{refG} A.\ Gimigliano. {\it On linear systems of plane curves}, 
Thesis, QueenÕs University, Kingston, (1987).

\bibitem[GuH]{refGuardoHar} E.\ Guardo and B.\ Harbourne. {\it Resolutions 
of ideals of six fat points in $\pr 2$}, J. Alg. 318 (2), 619--640 (2007).

\bibitem[H1]{refH1} B.\ Harbourne. {\it Birational models of rational surfaces}, 
J. Alg. 190, 145--162 (1997).

\bibitem[H2]{refH2} B.\ Harbourne. {\it Free Resolutions of Fat Point Ideals on $\pr 2$}, 
J. Pure Appl. Alg. 125, 213--234 (1998).


\bibitem[H3]{refH3} B.\ Harbourne. {\it An Algorithm for Fat Points on $\pr2$}, 
Can. J. Math. 52 (2000), 123--140.

\bibitem[H4]{refH4} B.\ Harbourne. {\it The geometry of rational surfaces and 
Hilbert functions of points in the plane}, Proccedings of the 1984 Vancouver 
Conference in Algebraic Geometry, 
CMS Conf. Proc., 6 Amer. Math. Soc., Providence, RI, (1986) 95--111.

\bibitem[H5]{refH5} B.\ Harbourne. 
{\it  Rational Surfaces with $K^2>0$}, Proc. Amer. Math. Soc. 124, 727--733 (1996).

\bibitem[H6]{refH6} B.\ Harbourne. 
{\it  Complete linear systems on rational 
surfaces}, Trans. Amer. Math. Soc. 289, 213--226 (1985).

\bibitem[H7]{refH7} B.\ Harbourne. 
{\it Global aspects of the geometry of surfaces}, in preparation.

\bibitem[Hr]{refHr} R.\ Hartshorne. {\it Algebraic Geometry}, Graduate texts in mathematics (52), New York,
Springer-Verlag, 496 pp., 1977.

\bibitem[Hi]{refHi} A.\ Hirschowitz. {\it Une conjecture pour la cohomologie des diviseurs sur les surfaces 
rationelles g\'en\'eriques}, J. Reine Angew. Math., 397 (1989), 208--213. 

\bibitem[Ho]{refHo} M.\ Hochster. {\it Criteria for equality of ordinary and 
symbolic powers of primes}, Math. Z. 1973, 133, 53--65.

\bibitem[HoH]{refHH1}
M.\ Hochster and C.\ Huneke. {\it Comparison of symbolic and ordinary powers of
ideals}, Invent. Math. {\bf 147} (2002), no.~2, 349--369.

\bibitem[HH]{refHH}
S.\ Huckaba and C.\ Huneke. {\it Powers of ideals having small analytic deviation}, Amer. J. Math. {\bf 114} (1992), 367--403.

\bibitem[LS]{refLS} A.\ Li and I.\ Swanson. {\it Symbolic powers of radical ideals}, 
Rocky Mountain J. of Math. 36 (2006), 997--1009.

\bibitem[M]{refM} H.\ Matsumura. {\it Commutative Algebra}. W. A. Benjamin, New York, (1970), pp. 212 + xii 

\bibitem[Mo]{refMo} S.\ Morey. {\it Stability of associated primes and 
equality of ordinary and symbolic powers of ideals}, Comm. Alg. 27(7), (1999), 3221--3231.

\bibitem[MNV]{refMNV} S.\ Morey, S.\ Noh and W.\ Vasconcelos. {\it Symbolic powers, Serre Conditions and 
Cohen-Macaulay Rees Algebras}. 
Manuscripta Math. 86 (1995), 113--124.

\bibitem[R]{refR} J.\ Ro\'e. {\it Limit linear systems and applications}, preprint.
(ArXiv: math.AG/0602213.pdf) 

\bibitem[S]{refS} B.\ Segre. {\it Alcune questioni su insiemi finiti di punti in geometria algebrica}, Atti Convegno Intern. di Geom. Alg. di Torino, (1961), 15--33. 

\bibitem[T]{refT} H.\ Terakawa. {\it The $d$-very ampleness on a 
projective surface in characteristic $p$}, Pac. J. Math. 187 (1999), 187--199.

\end{thebibliography}
\end{document}